\documentclass[11pt,leqno]{article}
\usepackage{amsmath,amssymb,mathrsfs,amsthm,bm,ddag}

\usepackage[normalem]{ulem}
\usepackage{eucal}

\usepackage{xy}
\input{xy}
\xyoption{all}

\usepackage[hypertex]{hyperref}

\usepackage{comment}

\makeatletter
\def\@seccntformat#1{\csname the#1\endcsname.\hspace{2ex}}

 \renewcommand{\subsection}%
  {\@startsection{subsection}%
  {2}%
  {\z@}%
  {2ex}
  {0ex}
  {\reset@font\normalsize\bfseries}}%

 \@addtoreset{equation}{subsection}

 \newcommand{\nsection}{\@startsection{section}{1}{\z@}%
     {-5ex}
     {1ex}
     {\reset@font\center\large\sc}}

 \renewenvironment{thebibliography}[1]
 {\nsection*{\refname\@mkboth{\refname}{\refname}}%
   \list{\@biblabel{\@arabic\c@enumiv}}%
   {\settowidth
   \labelwidth{\@biblabel{#1}}%
   \leftmargin
	\labelwidth
        \advance
	 \leftmargin
	 \labelsep
         \@openbib@code
         \usecounter{enumiv}%
         \let\p@enumiv\@empty
	 \parskip=0pt
	 \itemsep=1pt
	 \parsep=1pt
	 \itemindent=\z@
         \renewcommand\theenumiv{\@arabic\c@enumiv}}%
   	 \sloppy
   	 \clubpenalty4000
   	 \@clubpenalty\clubpenalty
   	 \widowpenalty4000%
   	 \footnotesize
   	 \sfcode`\.\@m}
  	 {\def\@noitemerr
    	 {\@latex@warning{Empty `thebibliography' environment}}%
   	 \endlist}

\makeatother

\newtheoremstyle{thm}
 {1em}
 {3pt}
 {\itshape}
 {}
 {\bf}
 {. ---}
 {0.5em}
 {}
\newtheoremstyle{dfn}
 {1em}
 {3pt}
 {}
 {}
 {\bf}
 {. {---}}
 {0.5em}
 {}

\swapnumbers
\theoremstyle{thm}

\newtheorem*{lem*}{Lemma}

\newtheorem*{cor*}{Corollary}
\newtheorem{prop}[subsection]{Proposition}
\newtheorem*{prop*}{Proposition}

\newtheorem*{conj*}{Conjecture}
\newtheorem*{thm*}{Theorem}

\theoremstyle{dfn}

\newtheorem*{dfn*}{Definition}

\newtheorem*{ex*}{Example}

\newtheorem*{rem*}{Remark}
\newtheorem*{cl}{Claim}

\usepackage{color}
\newenvironment{meta}{
\noindent \color{red}
\sffamily[}{\upshape]}

\newsavebox{\circlebox}
\savebox{\circlebox}{\fontencoding{OMS}\selectfont\char13}
\newlength{\circleboxwdht}

\renewcommand{\H}{\mc{H}}

\newcommand{\ul}[1]{\underline{#1}}

\newcommand{\Corr}{\mbf{Corr}}
\newcommand{\Cat}{\mbf{Cat}}
\newcommand{\reptd}{\ul{\mr{dt}}}
\newcommand{\repsw}{\ul{\mr{sw}}}
\newcommand{\Set}{\mc{S}\mr{et}}

\begin{document}
\title{On the Serre conjecture for Artin characters in the geometric case}
\author{Tomoyuki Abe}
\date{}
\maketitle

\begin{flushright}
 {\it Dedicated to Takeshi Saito with respect}
\end{flushright}

\begin{abstract}
 Let $G$ be a finite group and $A$ be a regular local ring on which $G$ acts.
 Under certain assumptions on $A$ and the action, Serre defined a function $a_G\colon G\rightarrow\mb{Z}$
 which can be viewed as a higher dimensional analogue of Artin character,
 and conjectured that it is associated to a $\mb{Q}_\ell$-rational representation of $G$ for any prime $\ell$ invertible in $A$.
 We prove this conjecture in the equal characteristic case.
\end{abstract}

\section*{Introduction}
Let $A$ be a regular local ring with maximal ideal $\mf{m}$ and $G$ be a finite group acting on $A$.
For $\sigma\in G$, let $I_\sigma$ be the ideal of $A$ generated by the set $\{a-\sigma(a)\mid a\in A\}$.
Assume that the following conditions are satisfied:
\begin{itemize}
 \item The $G$-fixed part $A^G$ is a noetherian ring and $A$ is finite over $A^G$;

 \item For any $\sigma\in G\setminus\{1\}$, the ring $A/I_\sigma$ is of finite length;

 \item The induced map $A^G/(A^G\cap\mf{m})\rightarrow A/\mf{m}$ is an isomorphism.
\end{itemize}
Under these assumptions, we define a function $a_G\colon G\rightarrow\mb{Z}$ by
\begin{equation*}
 a_G(\sigma):=
  \begin{cases}
   -\mr{length}_A(A/I_\sigma)&\qquad\sigma\in G\setminus\{1\}\\
   -\sum_{\xi\in G\setminus\{1\}} a_G(\xi)&\qquad\sigma=1.
  \end{cases}
\end{equation*}
When $A$ is a discrete valuation ring, $a_G$ is nothing but the Artin character (cf.\ \cite[\S19.1]{Slin}).
In this case, there exists a representation of $G$ whose associated character is $a_G$.
For any prime number $\ell$ invertible in $A$, this representation is $\mb{Q}_\ell$-rational.
See \cite[\S19.2, Thm 44]{Slin} for the details.
J.-P. Serre conjectured that similar phenomena occur for general $A$ as follows:

\begin{conj*}[Serre {\cite[\S6]{S}}]
 Let $\ell$ be a prime number which is invertible in $A$.
 \begin{enumerate}
  \item\label{conj-1}
       The function $a_G$ is the character of a $\overline{\mb{Q}}_\ell$-representation of $G$, which we call the {\em Artin representation}.

  \item\label{conj-2}
       The Artin representation is $\mb{Q}_{\ell}$-rational.
 \end{enumerate}
 
\end{conj*}
\noindent
In this article, we prove this conjecture in the equal characteristic situation:
\begin{thm*}
 \label{main}
 The conjecture is true in the case where $A$ contains a finite field.
\end{thm*}

Most of the known cases of the conjecture are in the case where $A$ is of dimension less than or equal to $2$.
The case $\dim(A)=1$ is classical.
The equal characteristic and $\dim(A)=2$ case was proved in \cite{KSS}.
In \cite{Ab}, the mixed characteristic case of $\dim(A)=2$ with some conditions is proved.
In \cite{KS}, conjecture \ref{conj-1} is reproved under the same conditions as \cite{KSS} in a more conceptual method.
Pursuing this method, in \cite{KS2}, conjecture \ref{conj-1} is proved for {\em any} ring $A$ whose dimension is $2$.
The author does not know how much is known in the case of $\dim(A)>2$.

Let us overview our proof.
In fact, our proof is surprisingly simple, and we do not use recent advances in ramification theory, for example the theory of characteristic cycles in \cite{Schar}.
We prove the theorem by solving \cite[Conj 5.1]{KSS}.
Roughly speaking, \cite[Conj 5.1]{KSS} states as follows:
Let $k$ be an algebraically closed field of characteristic $p>0$.
Let $S=\mr{Spec}(k\{t\})$, where $k\{t\}$ is the strict henselization of $\mb{A}^1_t:=\mr{Spec}(k[t])$ at $0$,
and $f\colon X\rightarrow S$ be a morphism of finite type and $G$ acts on $X$ over $S$ with single fixed point $x$.
Under some assumptions, \cite{KSS} conjectures that $\mr{Tr}(\sigma;\reptd(\mr{R}\phi_f\mb{Q}_{\ell})_x)=a_G(\sigma)$ for any $\sigma\in G\setminus\{1\}$ and $\ell\neq p$.
Here, $\mr{R}\phi_f$ is the vanishing cycle functor, and $\reptd$ denotes the associated ``total dimension representation'' (cf.\ Definition \ref{tddfn}).
Let $\mb{P}':=\mb{P}^1_{t'}$ (where $t'$ denotes the coordinate),
$\mr{pr}'\colon X\times\mb{P}'\rightarrow\mb{P}'$ be the projection, and $\mc{L}$ be the Artin-Schreier sheaf on $\mb{A}^1$ (associated with a fixed non-trivial additive character).
The pullback of $\mc{L}$ by the function $ft'$ is denoted by $\mc{L}(ft')$, which is a sheaf on $X\times\mb{P}'$.
The values of $\mc{L}(ft')$ at the points of the fiber of $t'=\infty'\in\mb{P}'$ will be important in the later discussion,
but we do not go into the details here and simply take the zero extension around the fiber $X\times\{\infty'\}$.
First, by using ideas from \cite[\S6]{A}, we interpret $\reptd(\mr{R}\phi_f\mb{Q}_{\ell})_x$ as $-[\Psi_{\mr{pr}'}(\mc{L}(ft'))_{(x,\infty')}]$.
This is a corollary of some computations in \cite{L}, and carried out in \S\ref{sec2}.

The computation of $\mr{Tr}(\sigma;\Psi_{\mr{pr}'}(\mc{L}(ft'))_{(x,\infty')})$ is done in \S\ref{sec1}.
For this, we use the Lefschetz-Verdier pairing introduced in [SGA 5, Exp.\ III].
Let $Y$ be a scheme, $\sigma$ be an endomorphism of $Y$, $\mc{F}$ be a sheaf on $Y$, and $\alpha\colon\sigma^*\mc{F}\rightarrow\mc{F}$.
With these data, Lefschetz-Verdier pairing yields an element $\tau(\mc{F},\alpha)=\left<\alpha,\mr{id}\right>\in\mr{H}^0(Y^{\sigma},\mc{K}_{Y^{\sigma}})$,
where $Y^{\sigma}$ is the ``$\sigma$-fixed part'' and $\mc{K}$ is the dualizing complex.
One of the key inputs of this article is the ``family version'' of Lefschetz-Verdier pairing established recently by Lu and Zheng \cite{LZ}.
They interpreted the classical Lefschetz-Verdier pairing as the categorical trace of a certain symmetric monoidal 2-category defined from $D^{\mr{b}}_{\mr{c}}(Y)$,
which enabled them to generalize this absolute picture of SGA naturally to the family situation.
More precisely, let $T$ be a scheme considered as a parameter space, and consider a sheaf $\mc{F}_T$ on $Y\times T$
with $\alpha_T\colon(\sigma\times\mr{id}_T)^*(\mc{F}_T)\rightarrow\mc{F}_T$.
When $\mc{F}_T$ is moreover {\em universally locally acyclic} over $T$, they defined the generalized Lefschetz-Verdier pairing which takes value in
$\mr{H}^0(Y^{\sigma}\times T,\mc{K}_{Y^{\sigma}\times T/T})$.
Using the fact that the pullback functor yields a symmetric monoidal functor between the symmetric monoidal 2-categories
they used for the construction,
the generalized Lefschetz-Verdier pairing is compatible with base change.

To complete our proof, we apply the formalism of generalized Lefschetz-Verdier pairing to our situation.
Assume that $\mc{L}(ft')$ is universally locally acyclic over $\mb{P}'$.
We can immediately observe that the restriction morphism
$\mr{H}^0(X^{\sigma}\times\mb{P}',\mc{K}_{X^{\sigma}\times\mb{P}'/\mb{P}'})\rightarrow\mr{H}^0(X^{\sigma}\times\{t'\},\mc{K}_{X^{\sigma}})\cong\Lambda$
is an isomorphism for any closed point $t'$ of $\mb{P}'$.
This implies that the Lefschetz-Verdier pairing does {\em not} change even if we vary $t'$ in $\mb{P}'$.
If we specialize $\mc{L}(ft')$ over $0'\in\mb{P}'$, we have $\mc{L}(ft')|_{X\times\{0\}}\cong\Lambda$.
Thus, the Lefschetz-Verdier pairing $\tau(\Lambda,\sigma)$, which is now classical, can be computed as a certain intersection product,
which in fact coincides with $a_G(\sigma)$.
If we specialize over $\infty'\in\mb{P}'$, we get the nearby cycle, and the conjecture follows.
In general $\mc{L}(ft')$ is not universally locally acyclic, and we need small additional arguments (cf.\ Theorem \ref{extula}).

The author first learned about the conjecture of Serre in the graduate course lecture given by Takeshi Saito
at the Komaba campus of the University of Tokyo in the winter semester of 2004.
Takeshi has always been generous enough to share wonderful problems and ideas with me,
without which my mathematical life would not have been so fruitful.
It is my great honor and pleasure to dedicate this article to Takeshi on the occasion of his 60th birthday, with admiration and respect.

\subsection*{Acknowledgment}\mbox{}\\
The author also wishes to thank Enlin Yang and Ofer Gabber:
The author learned the results of Lu and Zheng from Yang, and Gabber pointed out a gap in the proof of Theorem \ref{extula}.
He is grateful to the referee for reading the manuscript carefully and for valuable comments.
This work is supported by JSPS KAKENHI Grant Numbers 20H01790, 23K20202.

\subsection*{Conventions}
\begin{itemize}
 \item  In this article, we fix prime numbers $p$, $\ell$ coprime to each other.
	To simplify the presentation, we assume that all the schemes in this article are over $\mb{F}_p$.
	However, the results of \S\ref{sec1} except for Theorem \ref{thm} are valid without this assumption (with standard assumptions on the coefficients).

 \item  Assume we are given morphisms of schemes $X\rightarrow S$, $T\rightarrow S$.
	The base change $X\times_S T$ is often denoted by $X_T$.
	Likewise, if $\mc{F}$ is a sheaf (resp.\ complex, {\it etc.}) on $X$,
	the pullback to $X_T$ is denoted by $\mc{F}_T$,	if no confusion may arise.

 \item  For a scheme $X$ and its geometric point $x$, we denote by $X_{(x)}$ (the spectrum of) the strict henselization of $X$ at $x$.
\end{itemize}

\section{Construction of Artin representation}
\label{sec1}
\subsection{}
\label{absthrct}
One of the main tools of the proof is the relative Lefschetz-Verdier pairing of Lu and Zheng \cite{LZ}.
We recall their results very briefly.
We start by recalling some abstract nonsense.
Assume we are given a symmetric monoidal $2$-category $\mbf{C}$.
The formalism of categorical traces (cf.\ \cite[\S1]{LZ}) yields the trace functor
(between $(2,1)$-categories, namely $2$-categories whose $2$-morphisms are invertible)
\begin{equation*}
 \tau_{\mbf{C}}\colon\mr{End}(\mbf{C})\rightarrow\Omega\mbf{C}.
\end{equation*}
Here, $\mr{End}(\mbf{C})$ is a $(2,1)$-category whose objects consist of pairs $(X,e)$ where $X$ is a dualizable object of $\mbf{C}$ and $e\in\mr{End}_{\mbf{C}}(X)$.
The morphisms in $\mr{End}(\mbf{C})$ is not important for us.
See \cite[\S1.2]{LZ} for the details.
The category $\Omega\mbf{C}$ is the category (not just $(2,1)$-category) of endomorphisms $\mr{End}_{\mbf{C}}(\mbf{1}_{\mbf{C}})$.

The trace functor behaves functorially with respect to symmetric monoidal functors.
Let $F\colon\mbf{C}\rightarrow\mbf{D}$ be a symmetric monoidal functors between symmetric monoidal 2-categories.
Then, we have the following 2-commutative diagram of functors:
\begin{equation*}
 \xymatrix@C=50pt{
  \mr{End}(\mbf{C})\ar[r]^-{\tau_{\mbf{C}}}\ar[d]_{\mr{End}(F)}&\Omega\mbf{C}\ar[d]^{\Omega F}\\
 \mr{End}(\mbf{D})\ar[r]^-{\tau_{\mbf{D}}}&\Omega\mbf{D},
  }
\end{equation*}
which is remarked in \cite[Rem 1.9]{LZ}.

\subsection{}
\label{corprep}
Let us introduce 2-categories to which Lu and Zheng applied the above general theory.
Let $\Lambda$ be a torsion ring such that $m\Lambda=0$ for some integer $m$ which is prime to $p$.
Let $S$ be a noetherian scheme.
We consider the $2$-category $\Corr(S)$ of correspondences.
The objects are separated $S$-schemes of finite type.
A morphism $X\rightarrow Y$ in $\Corr(S)$ is a diagram $X\leftarrow Z\rightarrow Y$ of separated $S$-schemes of finite type.
We do not recall the definition of $2$-morphisms, and see \cite[\S 2.2]{LZ} for the details.
Let $\Cat$ be the $2$-category of (small) categories.
We have the functor $F_S\colon\Corr(S)^{\mr{2-op}}\rightarrow\Cat$ such that for $X\in\Corr(S)$, $F_S(X)$
is defined to be $D_{\mr{ctf}}(X,\Lambda)$.
For a correspondence $c\colon X\xleftarrow{f}Z\xrightarrow{g}Y$, we put $F_S(c):=g_!f^*\colon D(X)\rightarrow D(Y)$.
Let $\mc{C}_{S,\Lambda}\rightarrow\Corr(S)$ be the Grothendieck (op)fibration corresponding to the functor (cf.\ \cite[\S1.3]{LZ}).
Concretely, objects of $\mc{C}_{S,\Lambda}$ are the pairs $(X,\mc{F})$ where $X$ is an $S$-scheme and $\mc{F}\in D_{\mr{ctf}}(X,\Lambda)$.
If no confusion may arise, we denote $\mc{C}_{S,\Lambda}$ simply by $\mc{C}_S$.

The category $\mc{C}_{S,\Lambda}$ admits a symmetric monoidal structure given by $(X,\mc{F})\otimes(Y,\mc{G}):=(X\times_S Y,\mc{F}\boxtimes_S\mc{G})$.
Assume we are given a morphism of noetherian schemes $g\colon T\rightarrow S$.
As in the proof \cite[Prop 2.26]{LZ}, we have the pullback functor $g^*\colon\mc{C}_S\rightarrow\mc{C}_T$ sending $(X,\mc{F})$ to $(X\times_ST,g^*\mc{F})$,
which is also shown to preserve symmetric monoidal structures.

\subsection{}
After these preparations, we may apply the general construction of \ref{absthrct}
to the symmetric monoidal 2-category $\mc{C}_S=\mc{C}_{S,\Lambda}$ introduced in \ref{corprep}.
This yields the trace functor $\tau_S\colon\mr{End}(\mc{C}_S)\rightarrow\Omega\mc{C}_S$ for a noetherian scheme $S$.
Furthermore, assume we are given a morphism of noetherian schemes $g\colon T\rightarrow S$.
Since the functor $g^*\colon\mc{C}_S\rightarrow\mc{C}_T$ preserves symmetric monoidal structure as we discussed in \ref{corprep},
we have the following 2-commutative diagram by \ref{absthrct}:
\begin{equation*}
 \xymatrix@C=50pt{
  \mr{End}(\mc{C}_{S,\Lambda})\ar[r]^-{\tau_S}\ar[d]_{\mr{End}(g^*)}&\Omega\mc{C}_{S,\Lambda}\ar[d]^{\Omega g^*}\\
 \mr{End}(\mc{C}_{T,\Lambda})\ar[r]^-{\tau_T}&\Omega\mc{C}_{T,\Lambda}.
  }
\end{equation*}
In order to interpret this diagram using a more familiar language, let us have a closer look at the objects appearing in the diagram.
Giving an object of $\mr{End}(\mc{C}_{S,\Lambda})$ is equivalent to giving a dualizable object $(X,\mc{F})$ and an endomorphism
$(z,w)\colon(X,\mc{F})\rightarrow(X,\mc{F})$ in $\mc{C}_{S,\Lambda}$.
More explicitly, this is equivalent to giving a correspondence $z\colon X\xleftarrow{f}Z\xrightarrow{g}X$,
a dualizable object $\mc{F}\in D_{\mr{ctf}}(X,\Lambda)$, and a homomorphism $w\colon f^*\mc{F}\rightarrow g^!\mc{F}$ in $D_{\mr{ctf}}(Z,\Lambda)$.
Furthermore, an object of $\Omega\mc{C}_{S,\Lambda}$ can be regarded as a pair $(Y,\alpha)$ where $Y\in\Corr(S)$ and
\begin{equation*}
 \alpha\in\mr{H}^0(Y,\mc{K}_{Y/S}):=
  \mr{Hom}_{D(Y,\Lambda)}\bigl(\pi^*\Lambda_S,\pi^!\Lambda_S\bigr),
\end{equation*}
where $\pi\colon Y\rightarrow S$ is the structural morphism.
Assume we are given an object $(z,w)$ of $\mr{End}(\mc{C}_{S,\Lambda})$ as above.
By the concrete description \cite[Rem 2.15]{LZ},
the categorical trace $\tau_S(z,w)$ actually belongs to $\mr{H}^0(X^z,\mc{K}_{X^z/S})$ where $X^z:=Z\times_{(f,g),X\times X,\Delta_X}X$.

Up till now, the consequences follow more or less formally once the formalism is set up properly.
However, these results are for dualizable objects of $\mc{C}_{S,\Lambda}$, which are defined abstractly.
One of the surprising discoveries of Lu and Zheng is that the dualizable objects in $\mc{C}_{S,\Lambda}$ are in fact the complexes that we are familiar with.
Namely, an object $(X,\mc{F})$ in $\mc{C}_{S,\Lambda}$ (note in particular that $\mc{F}\in D_{\mr{ctf}}(X,\Lambda)$)
is dualizable exactly when $\mc{F}$ is locally acyclic\footnote{
It is more common to say that the morphism $h\colon X\rightarrow S$ is locally acyclic with respect to $\mc{F}$
(cf.\ [SGA $4\frac{1}{2}$, Th.\ finitude, Def 2.12]), but in this article, we say that $\mc{F}$ is $h$-locally acyclic.
We also say that $\mc{F}$ is locally acyclic over $S$ if no confusion may arise.}
over $S$ by \cite[Thm 2.16]{LZ}.
What we have observed can be summarized as follows:

\begin{thm*}
 \label{bc}
 Recall that $\Lambda$ is a torsion ring such that $m\Lambda=0$ for some integer $m$ which is prime to $p$.
 Let $g\colon T\rightarrow S$ be a morphism between noetherian schemes,
 and assume we are given an endomorphism $(z,w)\colon(X,\mc{F})\rightarrow(X,\mc{F})$ in $\mc{C}_{S,\Lambda}$ such that $\mc{F}$ is {\em locally acyclic} over $S$.
 Then we have $g^*\tau_S(z,w)=\tau_T\bigl(g^*(z,w)\bigr)$ where $g^*$ are the ``base change'' homomorphisms
 \begin{equation*}
  g^*\colon\mr{H}^0(X^z,\mc{K}_{X^z/S})\rightarrow\mr{H}^0(X^z_T,\mc{K}_{X^z_T/T}),\qquad
   g^*\colon\mr{End}(\mc{C}_{S,\Lambda})\rightarrow\mr{End}(\mc{C}_{T,\Lambda}).
 \end{equation*}
\end{thm*}

\subsection{}
\label{limlem}
We need the following (small) generalization of [SGA 4, Exp.\ IX, Cor 2.7.4].
This should be well-known to experts, but we could not find a suitable reference,
so we include a proof for the sake of completeness.
\begin{lem*}
 Assume that $\Lambda$ is a noetherian ring.
 Let $I$ be a {\normalfont(}small{\normalfont)} filtered category,
 and consider a functor $X_{\bullet}\colon I^{\mr{op}}\rightarrow\mr{Sch}^{\mr{coh}}$,
 where $\mr{Sch}^{\mr{coh}}$ is the category of coherent schemes {\normalfont(}{i.e.\ }quasi-compact quasi-separated schemes{\normalfont)}.
 Assume that for any morphism $i\rightarrow j$, $X_j\rightarrow X_i$ is affine.
 Put $X_\infty:=\invlim X_\bullet$.
 We have a canonical equivalence
 \begin{equation*}
  \rho\colon
  2\mr{-}\indlim D^{\mr{b}}_{\mr{c}}(X_i,\Lambda)\xrightarrow{\sim}
   D_{\mr{c}}^{\mr{b}}(X_{\infty},\Lambda).
 \end{equation*}
\end{lem*}
\begin{proof}
 We denote $D(Z,\Lambda)$ by $D(Z)$.
 First, let us show the full faithfulness.
 We take $\mc{F},\mc{G}$ in $D^{\mr{b}}_{\mr{c}}(X_i)$ for some $i\in I$, and we must show that the canonical homomorphism
 \begin{equation*}
  \indlim_{j\in I_{/i}}\mr{Hom}_{D(X_j)}(\pi_j^*\mc{F},\pi_j^*\mc{G})
   \rightarrow
   \mr{Hom}_{D(X_\infty)}(\pi_{\infty}^*\mc{F},\pi_{\infty}^*\mc{G})
 \end{equation*}
 is an isomorphism where $\pi_j\colon X_j\rightarrow X_i$ is the structural morphism.
 In fact, we show this for any $\mc{F}\in D^-_{\mr{c}}(X_i)$ and $\mc{G}\in D^+(X_i)$.
 We may replace $\mc{F}$ by $\tau^{>(-N)}\mc{F}$ for a large enough $N$ because $\pi_{\star}^*$, where $\star\in\{j,\infty\}$, is exact and $\mc{G}$ is bounded below.
 By shifting $\mc{G}$, we may assume that $\H^i\mc{F}=0$ for $i\not\in[0,n]$ for some $n\geq0$.
 By considering the exact triangle $\tau^{<n}\mc{F}\rightarrow\mc{F}\rightarrow\H^n\mc{F}[-n]\rightarrow$ and using the induction on $n$,
 we may assume that $\mc{F}$ is a constructible $\Lambda$-module and $\mc{G}$ is a shift of a $\Lambda$-module.
 Thus, it suffices to show that the canonical homomorphism
 \begin{equation*}
  \indlim_{j\in I_{/i}}\mr{Ext}_{\Lambda_{X_j}}^n(\pi_j^*\mc{F},\pi_j^*\mc{G})
   \rightarrow
   \mr{Ext}_{\Lambda_{X_\infty}}^n(\pi_{\infty}^*\mc{F},\pi_{\infty}^*\mc{G})
 \end{equation*}
 is an isomorphism for any $n$, constructible $\Lambda$-module $\mc{F}$, and $\Lambda$-module $\mc{G}$.
 Since $\mc{F}$ is constructible, there exist a quasi-compact \'{e}tale morphism $j\colon U\rightarrow X_i$, a constructible $\Lambda$-module $\mc{K}$,
 and a short exact sequence $0\rightarrow\mc{K}\rightarrow j_!\Lambda_U\rightarrow\mc{F}\rightarrow0$ by [SGA 4, Exp.\ IX, Prop 2.7].
 By using the induction on $n$, we are reduced to showing the isomorphism for $\mc{F}=j_!\Lambda_U$.
 By adjunction $(j_!,j^*)$, the claim follows by [SGA 4, Exp.\ VI, Cor 8.7.9], and thus $\rho$ is fully faithful.

 Let us show the essential surjectivity.
 Take $\mc{F}\in D^{\mr{b}}_{\mr{c}}(X_\infty)$.
 There exist integers $b\leq a$ such that $\mc{H}^i\mc{F}=0$ for $i\not\in[b,a]$.
 Without loss of generality, we may assume that $b=0$.
 We use the induction on $a$.
 If $a=0$, this is a consequence of [SGA 4, Exp.\ IX, Cor 2.7.4].
 For general $a$, consider the exact triangle $\tau^{<a}\mc{F}\rightarrow\mc{F}\rightarrow(\mc{H}^a\mc{F})[-a]\rightarrow$.
 The connecting morphism yields an element $e_{\mc{F}}\in\mr{Hom}_{D(X_\infty)}\bigl((\mc{H}^a\mc{F})[-a-1],\tau^{<a}\mc{F}\bigr)$,
 and the exact triangle induces an isomorphism $\mr{Cone}(e_{\mc{F}})\cong\mc{F}$.
 By induction hypothesis, there exist $\mc{G}$, $\mc{G}'\in D^{\mr{b}}_{\mr{c}}(X_i)$
 and isomorphisms $\pi_{\infty}^*\mc{G}\cong\tau^{<a}\mc{F}$, $\pi_{\infty}^*\mc{G}'\cong(\mc{H}^a\mc{F})[-a]$ for some $i$.
 By the full faithfulness of $\rho$ that we have already proven, there exists $j\in I_{/i}$ and a homomorphism
 $\phi\colon\pi_j^*\mc{G}'[-1]\rightarrow\pi_j^*\mc{G}$ whose pullback to $X_\infty$ is isomorphic to $e_{\mc{F}}$.
 Let $\pi'\colon X_{\infty}\rightarrow X_j$ be the morphism.
 Then we have $\pi'^*\mr{Cone}(\phi)\cong\mr{Cone}(\pi'^*\phi)\cong\mr{Cone}(e_{\mc{F}})\cong\mc{F}$.
 Thus, $\mc{F}$ belongs to the essential image of $\rho$ as required.
\end{proof}

\subsection{}
The sheaf $\mc{L}(ft')$ appeared in Introduction is usually {\em not} locally acyclic as it is.
This prevents us from applying Theorem \ref{bc} directly to $\mc{L}(ft')$.
To remedy this, we need the following theorem, which roughly asserts that if we are given a morphism of finite type $f\colon X\rightarrow S$,
and a constructible sheaf $\mc{F}$ on the generic fiber of $f$, there exists an alteration $T$ of $S$ so that $\mc{F}$
extends to a constructible complex on $X_T$ which is universally locally acyclic.

\begin{thm*}
 \label{extula}
 Let $\Lambda$ be a finite ring in which $p$ is invertible,
 $S$ a coherent scheme with finitely many irreducible components, and $U\subset S$ be an open subscheme.
 Let $f\colon X\rightarrow S$ be a morphism of finite presentation,
 and let $\mc{F}$ be an object of $D_{\mr{ctf}}(X_U,\Lambda)$ which is $f_U$-universally locally acyclic.
 Let $\sigma\colon X\rightarrow X$ be an $S$-endomorphism,
 and assume that we are given a homomorphism $\alpha\colon\sigma_U^*\mc{F}\rightarrow\mc{F}$.
 Then there exists a surjective alteration $T\rightarrow S$, a complex $\mc{G}\in D_{\mr{ctf}}(X_T,\Lambda)$,
 and a homomorphism $\beta\colon\sigma_T^*\mc{G}\rightarrow\mc{G}$ such that the following hold:
 \begin{itemize}
  \item The constructible complex $\mc{G}$ is $f_T$-universally locally acyclic.

  \item Let $V:=T\times_S U$. Then $\mc{G}$ and $\beta$ coincides with $\mc{F}_V$ and $\alpha_V$.
 \end{itemize}
\end{thm*}

Since the proof is rather technical, let us explain some ideas behind the proof for the convenience of the reader.
The fundamental idea, which goes back to Deligne, is that when we take a large enough ``cover'', then pathological behaviors of nearby cycles in family disappear.
Following this philosophy, we take the inverse limit $\overline{S}$ over all such covers, and show the lemma over such a ``huge'' space $\overline{S}$.
In our situation, $\overline{S}$ becomes an absolutely integrally closed scheme ({\em i.e.}, a normal scheme whose field of fraction is algebraically closed),
and we first show the lemma in this case.
This case can be treated by simple applications of results of Orgogozo.
In this case, a very similar result has already appeared in \cite[Thm 4.1]{HS}.
By another application of Orgogozo's result, we may show that the extension is in fact constructible and descend to get the desired extension
for a usual (more precisely, finite type) cover of $S$.

\begin{proof}
 For a topos $\mc{T}$, we omit $D(\mc{T},\Lambda)$ by $D(\mc{T})$.
 By considering the irreducible components, we may assume that $S$ is integral with an {\em algebraic} geometric generic point $\xi$
 by the assumption that $S$ has finitely many irreducible components.
 By replacing $S$ by an alteration, we may assume that $\mc{F}_!:=(X_U\hookrightarrow X)_!(\mc{F})$,
 is $f$-good (in the sense of Illusie \cite[Def 1.5]{I}) by \cite[Thm 2.1]{O}.
 Let $\overline{S}$ be the normalization of $S$ in $\xi$, and put $\overline{U}:=U\times_S\overline{S}$.
 Note that $\overline{S}$ is coherent as well.
 Let $\overline{f}:=f\times_S\overline{S}$ and $j\colon X_\xi\rightarrow X_{\overline{S}}$ be the canonical morphism.

 First, we claim that $\mr{R}j_*\mc{F}_{\xi}$ is constructible.
 Indeed, let $\overline{\mc{F}}$ be the pullback of $\mc{F}$ on $X_{\overline{U}}$, and let $\overline{\mc{F}}_!$ its zero-extension to $X_{\overline{S}}$.
 Then the sheaf $\overline{\mc{F}}_!$ is constructible.
 By invoking \cite[Thm 8.1]{O}, we may find a modification $Z\rightarrow\overline{S}$ so that
 $(\overline{\mc{F}}_!)_Z$ is $f_Z$-good and $\Psi_{f_Z}((\overline{\mc{F}}_!)_Z)$ is constructible in $D^{\mr{b}}(X_Z\overleftarrow{\times}_ZZ)$.
 Note that $Z$ is also coherent.
 Thus, the restriction $\Psi_{f_Z}((\overline{\mc{F}}_!)_Z)|_{X_Z\overleftarrow{\times}_Z\xi}$ is constructible as well.
 Let $\overline{Z}\rightarrow Z$ be the normalization of $Z$ in $\xi$.
 Then the canonical morphism of topoi $\pi\colon X_{\overline{Z}}\overleftarrow{\times}_{\overline{Z}}\xi\rightarrow X_{\overline{Z}}$
 is an equivalence because $\pi$ induces a bijection on the set of (equivalence classes of) points of topoi,
 $(\pi_*\mc{H})_x\cong\mc{H}_{(x,\xi)}$ for any sheaf $\mc{H}$ on $X_{\overline{Z}}\overleftarrow{\times}_{\overline{Z}}\xi$
 and any point $x$ of $X_{\overline{Z}}$ by \cite[Exp.\ XI, Cor 2.3.2]{G}, and the topoi have enough points by \cite[\S9.1]{O}.
 By $f_{Z}$-goodness, the pullback of $\Psi_{f_Z}((\overline{\mc{F}}_!)_Z)$ to $X_{\overline{Z}}\overleftarrow{\times}_{\overline{Z}}\xi$ is isomorphic to
 $\Psi_{f_{\overline{Z}}}((\overline{\mc{F}}_!)_{\overline{Z}})|_{X_{\overline{Z}}\overleftarrow{\times}_{\overline{Z}}\xi}$,
 and the latter is further isomorphic to $\mr{R}\overline{j}_*\mc{F}_\xi$
 where $\overline{j}\colon X_\xi\rightarrow X_{\overline{Z}}$.
 Summing up, $\mr{R}\overline{j}_*\mc{F}_\xi$ is constructible.
 Now, since $\mc{F}_!$ is $f$-good, the canonical morphism $r^*\mr{R}j_*\mc{F}_\xi\rightarrow\mr{R}\overline{j}_*\mc{F}_\xi$
 is an isomorphism where $r\colon X_{\overline{Z}}\rightarrow X_{\overline{S}}$ is the canonical morphism.
 Thus, $r^*\mr{R}j_*\mc{F}_\xi$ is constructible.
 Since $r$ is surjective on geometric points and $X_{\overline{Z}}$ and $X_{\overline{S}}$ are coherent,
 $\mr{R}j_*\mc{F}_\xi$ is constructible as well (cf.\ \cite[\S9.1]{O}).
 Furthermore, it belongs to $D_{\mr{ctf}}(X_{\overline{S}})$ by [SGA 4, Exp.\ XVII, Thm 5.2.11].

 Since $\mr{R}j_*\mc{F}_{\xi}=:\mc{E}$ is constructible, by \cite[Thm 2.1 and Thm 7.1]{O},
 we may find an alteration $T\rightarrow\overline{S}$ such that $\mc{E}_T$ is $f_T$-good and $f_T$-cohomologically proper around each geometric point of $X_T$.
 Let us show that $\mc{E}_T$ is $f_T$-locally acyclic.
 Let $\overline{T}$ be an absolute integral closure of $T$ with generic point $\xi_{\overline{T}}$.
 Let $j_{\overline{T}}\colon X_{\xi_{\overline{T}}}\rightarrow X_{\overline{T}}$.
 Since $\mc{F}_!$ is $f$-good, we have the canonical isomorphism
 $\mc{E}_{\overline{T}}:=r'^*\mr{R}j_*\mc{F}_{\xi}\cong\mr{R}j_{\overline{T}*}\mc{F}_{\xi_{\overline{T}}}$ where $r'\colon X_{\overline{T}}\rightarrow X_{\overline{S}}$.
 We first prove that $\mc{E}_{\overline{T}}$ is locally acyclic over $\overline{T}$.
 Let $\zeta\rightsquigarrow t$ be a specialization map of geometric points on $\overline{T}$ and $x$ be a geometric point of $X_{\overline{T}}$ over $t$.
 We must show that the composite of the canonical morphisms
 \begin{equation*}
  \mc{E}_x\cong\mr{R}\Gamma\bigl((X_{\overline{T}})_{(x)},\mc{E}_{\overline{T}}\bigr)
   \xrightarrow{a}
   \mr{R}\Gamma\bigl((X_{\overline{T}})_{(x)}\times_{\overline{T}_{(t)}}\overline{T}_{(\zeta)},\mc{E}_{\overline{T}}\bigr)
   \xrightarrow{b}
   \mr{R}\Gamma\bigl((X_{\overline{T}})_{(x)}\times_{\overline{T}_{(t)}}\zeta,\mc{E}_{\overline{T}}\bigr)
 \end{equation*}
 is an isomorphism.
 Since $\overline{T}$ is absolutely integrally closed, the strict henselization coincides with the (Zariski) localization,
 and $a$ is an isomorphism since $\mc{E}_{\overline{T}}\cong\mr{R}j_{\overline{T}*}\mc{F}_{\xi_{\overline{T}}}$.
 The morphism $b$ is an isomorphism by the cohomological properness of $(X_{\overline{T}})_{(x)}\rightarrow\overline{T}_{(t)}$ with respect to $\mc{E}_{\overline{T}}$,
 and the $f_{\overline{T}}$-local acyclicity of $\mc{E}_{\overline{T}}$ follows.
 This in fact implies that $\mc{E}_T$ is $f_T$-locally acyclic.
 Indeed,
 since $\overline{T}\rightarrow T$ is integral, we have $(X_{\overline{T}})_{(x)}\cong(X_{T})_{(x)}\times_{T_{(t)}}\overline{T}_{(t)}$ by [EGA IV, 18.8.11].
 This implies that the ``Milnor fiber'' of $T$ at $\zeta$ and that of $\overline{T}$ at $\zeta$ coincide.
 Moreover, any specialization map on $T$ lifts to that on $\overline{T}$ since the morphism is integral and surjective, thus the local acyclicity of $\mc{E}_T$ follows.
 Since $\mc{E}_T$ is $f_T$-good, this implies that $\mc{E}_T$ is in fact $f_T$-universally locally acyclic by \cite[Ex 1.7 (b)]{I}.

 Let us finish the proof.
 By limit argument, there exists a finite morphism $S'\rightarrow S$ which factors $\overline{S}\rightarrow S$ and a surjective alteration $T'\rightarrow S'$
 such that $T'\times_{S'}\overline{S}\cong T$.
 Since $\mc{E}(:=\mr{R}j_*\mc{F}_\xi)$ is constructible, there exists a finite surjective morphism $T_0\rightarrow T'$
 which factors $T\rightarrow T'$ so that the $f_T$-universally locally acyclic complex $\mc{E}_T$ descends over $X_{T_0}$ by Lemma \ref{limlem}.
 Let $\mc{G}_0\in D^{\mr{b}}_{\mr{c}}(X_{T_0})$ be a descent.
 Since $T\rightarrow T_0$ is surjective and integral, $\mc{G}_0$ is $f_{T_0}$-universally locally acyclic by the coincidence of Milnor fibers of $f_{T_0}$ with that of $f_T$.
 Moreover, since $T\rightarrow T_0$ is surjective, $\mc{G}$ is automatically of finite Tor-dimension, and in particular, belongs to $D_{\mr{ctf}}$.
 The canonical morphism $\phi\colon\mc{F}_{\overline{T}_U}\rightarrow\mr{R}j_*\mc{F}_{\xi_{\overline{T}}}|_{X_{\overline{T}_U}}$ is an isomorphism
 since $\mc{F}_{\overline{T}_U}$ is locally acyclic.
 By another application of Lemma \ref{limlem}, there exists a finite surjective morphism $T_1\rightarrow T_0$ over which the isomorphism $\phi$
 descends to an isomorphism $(\mc{F})_{X_{T_1}\times_SU}\xrightarrow{\sim}(\mc{G}_0)_{X_{T_1}\times_SU}$.
 Likewise, $\alpha$ induces a homomorphism
 $\beta\colon\sigma^*_{\overline{U}}\mr{R}j_*\mc{F}_\xi|_{X_{\overline{U}}}\rightarrow\mr{R}j_*\mc{F}_\xi|_{X_{\overline{U}}}$
 compatible with $\phi$, and there exists a finite surjective morphism $T_2\rightarrow T_1$ over which $\beta$ descends.
 By putting $\mc{G}:=(\mc{G}_0)_{T_2}$, the conditions of the lemma are met.
\end{proof}

\begin{rem*}
 In the setting of the theorem, assume that $\mc{F}_\eta$ is $f$-good (cf.\ \cite[4.3]{A}) for any geometric generic point $\eta$ of $S$.
 Then, in fact, we may take the alteration $T\rightarrow S$ to be a finite morphism.
 This will be discussed in \cite{AG}.
\end{rem*}

\subsection{}
\label{mainthmcons}
Let $\Lambda$ be one of the following rings:
\begin{itemize}
 \item a ring such that $m\Lambda=0$ for some integer $m$ which is prime to $p$;

 \item a local regular finite $\mb{Z}_{\ell}$-algebra;

 \item a finite extension of $\mb{Q}_{\ell}$.
\end{itemize}
We fix a non-trivial additive character $\psi\colon\mb{F}_p\rightarrow\Lambda^{\times}$,
and let $\mc{L}_\Lambda$ be the Artin-Schreier sheaf on $\mr{Spec}(\mb{F}_p[t])$ associated with $\psi$.
We abbreviate $\mc{L}_{\mb{Q}_\ell(\zeta_p)}$ by $\mc{L}$ for simplicity.
When a base field $k$ of characteristic $p(>0)$ is fixed, we denote by $\mb{A}^1_x$ (resp.\ $\mb{P}^1_x$) the affine (resp.\ projective) line with coordinate $x$ over $k$.
We denote the composite $\mb{A}^1_x\times\mb{A}^1_y\rightarrow\mb{A}_t^1\rightarrow\mb{A}^1_{/\mb{F}_p}:=\mr{Spec}(\mb{F}_p[t])$ by $\mu$,
where the first morphism is the multiplication morphism, namely the morphism sending $(x,y)$ to $t=xy$, and second morphism is the canonical morphism.

Furthermore, let $G$ be a finite group, $X$ be a noetherian scheme on which $G$ acts, $x\in X$ be a closed point,
and assume that for each $\sigma\in G\setminus\{1\}$ the underlying set of $X^\sigma$ is equal to $\{x\}$.
Under this setup, we define $a_{X,G}(\sigma):=-\mr{length}_{\mc{O}_{X,x}}(\mc{O}_{X^{\sigma}})$,
and $a_{X,G}(1):=-\sum_{\xi\in G\setminus\{1\}}a_{X,G}(\xi)$.

\begin{thm*}
 \label{thm}
 Assume that $k$ is algebraically closed.
 Put $S:=\mb{P}^1_{t,(0)}$, $S':=\mb{P}^1_{t',(\infty')}$.
 Let $f\colon X\rightarrow S$ be a flat morphism of finite type, $x\in X$ a closed point,
 and let $G$ be a finite group which acts on $X$ by $S$-automorphisms.
 We further assume that $X$ is regular, $\bigcup_{\sigma\in G\setminus\{1\}} X^{\sigma}=\{x\}$ as underlying sets, and $f|_{X\setminus\{x\}}$ is smooth.

 We have the morphism $ft'\colon X\times \eta'\xrightarrow{f\times\mr{id}} S\times \eta'\rightarrow\mb{A}^1_t\times\mb{A}^1_{t'}\xrightarrow{\mu}\mb{A}^1_{/\mb{F}_p}$,
 where $\eta'$ is the generic point of $S'$, and $\mr{pr}'\colon X\times S'\rightarrow S'$ is the projection.
 We denote $j'_!(ft')^*\mc{L}$ by $\mc{L}(ft')$, where $j'\colon X\times\eta'\hookrightarrow X\times S'$ is the open immersion.
 Then, for $\sigma\in G\setminus\{1\}$, we have
 \begin{equation}
  \label{nbcyag}
  -a_{X,G}(\sigma)=
   \mr{Tr}\bigl(\sigma;\Psi_{\mr{pr}'}(\mc{L}(ft'))_{(x,\infty')}\bigr)
 \end{equation}
 in $\mb{Q}_\ell(\zeta_p)$, where $\Psi_{\mr{pr}'}$ is the nearby cycle functor\footnote{This is the nearby cycle functor denoted by $\mr{R}\Psi$ in [SGA 7, Exp.\ XIII].}
 with respect to $\mr{pr}'$.
\end{thm*}

\begin{proof}
 We may write $S=\invlim_{i\in I}S_i$ where $S_i$ is an \'{e}tale neighborhood of $\mb{A}^1$ around $0$, $I$ is a filtered category,
 and the transition morphisms are affine.
 By changing $I$ if needed, we may find a morphism of projective systems $\{X_i\}_{i\in I}\rightarrow\{S_i\}_{i\in I}$ such that
 $X=\invlim X_i$ and $X_j\xrightarrow{\sim}X_i\times_{S_i}S_j$ for any morphism $j\rightarrow i$ in $I$.
 By [EGA IV, 8.8.2], since $G$ is finite, we may find $i\in I$ and an action of $G$ on $X_i$ over $S_i$ which is compatible with the action on $X$ over $S$.
 By replacing $I$ by $I_{/i}$, we may assume that $G$ acts on $X_i$ for any $i\in I$.
 By further changing $I$, we may also assume that $S_i\times_{\mb{A}^1}\{0\}$ consists of a single point.
 Now, we have $X^{\sigma}=\invlim(X_i)^{\sigma}$.
 We claim that there exists $i\in I$ such that $X^{\sigma}_i$ consists of a single point $x_i$.
 Indeed, let $U_i:=S_i\times_{\mb{A}^1}(\mb{A}^1\setminus\{0\})$.
 Then $\invlim U_i\cong\eta$, where $\eta$ is the open point, in other words the generic point, of $S$.
 It suffices to find $i$ so that $X^{\sigma}_i\times_{S_i}U_i=\emptyset$.
 This follows by [EGA IV, 8.3.6].
 Similarly, we may find $i$ so that the morphism $X_i\rightarrow S_i$ is smooth outside of the closed point $x_i$.
 Since $S$ is the (strict) henselization of $S_i$ at the image of $0$, the nearby cycle remains the same.
 On the other hand, we have $a_{X,G}(\sigma)=a_{X_i,G}(\sigma)$ since
 \begin{equation*}
  (X_i)^{\sigma}
   \cong
   (X_i)^{\sigma}\times_{S_i}\{s_i\}
   \cong
   (X_i)^{\sigma}\times_{S_i}\{s\}
   \cong
   (X_i)^{\sigma}\times_{S_i}S\times_S\{s\}
   \cong
   X^{\sigma},
 \end{equation*}
 where $s_i$ (resp.\ $s$) is the closed point of $S_i$ (resp.\ $S$) over which $x_i$ (resp.\ $x$) lies.
 Thus, it suffices to show the following (slightly) more geometric claim:

 \begin{cl}
  Let $S$ be an \'{e}tale neighborhood in $\mb{A}^1=\mr{Spec}(k[t])$ of $0$, $f\colon X\rightarrow S$
  be a flat morphism between smooth schemes of finite type over $k$.
  Assume that a finite group $G$ acts on $X$ over $S$.
  Assume that there is a closed point $x\in X$ such that for each $\sigma\in G\setminus\{1\}$,
  the underlying set of the fixed scheme $X^{\sigma}$ is $\{x\}$, and $f|_{X\setminus\{x\}}$ is smooth.
  Let $\mr{pr}'\colon X\times\mb{P}^1_{t'}\rightarrow\mb{P}^1_{t'}$ be the second projection,
  $ft'\colon X\times\mb{A}^1_{t'}\xrightarrow{f\times\mr{id}} S\times\mb{A}^1_{t'}\rightarrow\mb{A}^1_t\times\mb{A}^1_{t'}\xrightarrow{\mu}\mb{A}^1$,
  and $\mc{L}(ft')$ be the zero-extension to $X\times\mb{P}^1_{t'}$ of $(ft')^*\mc{L}$.
  Then we have the equality (\ref{nbcyag}).
 \end{cl}
 Let us show the claim.
 Let $\Lambda$ be one of the 3 types of rings at the beginning of \ref{mainthmcons}.
 First, we note that $\Psi_{\mr{pr}'}(\mc{L}_{\Lambda}(ft'))_{(x,\infty')}$ is a perfect $\Lambda$-complex.
 Indeed, when $\Lambda$ is torsion, $\Psi_{\mr{pr}'}(\mc{L}_{\Lambda}(ft'))$ belongs to $D_{\mr{ctf}}(X\times\infty',\Lambda)$ by
 \cite[Rem 8.3]{O} and the perfectness follows.
 Even when $\Lambda$ is not a torsion, we can argue similarly to \cite[Rem 8.3]{O} since the cohomological dimension of $\invlim_n$ is $1$.
 Now, it suffices to show the equality (\ref{nbcyag}) for $\mc{L}(ft')$ replaced by $\mc{L}_{A}(ft')$,
 where $A:=\mb{Z}_\ell[\zeta_p]$, since the trace (of the perfect $A$-complex) remains the same.
 Since
 \begin{equation*}
  \mr{Tr}\bigl(\sigma;\Psi_{\mr{pr}'}(\mc{L}_A(ft'))_{(x,\infty')}\bigr)=\bigl\{\mr{Tr}(\dots(\mc{L}_{A/\ell^n}(ft'))\dots)\bigr\}_{n},
 \end{equation*}
 we may further replace $\mc{L}_A(ft')$ by $\mc{L}_{A/\ell^n}(ft')$.
 From now on, we fix $n>0$ and $\sigma\in G\setminus\{1\}$, and put $\Lambda:=A/\ell^n$.
 By Theorem \ref{extula}, we may take a finite surjective morphism $\widetilde{\mb{P}}'\rightarrow\mb{P}^1_{t'}$ such that $\widetilde{\mb{P}}'$ is {\em integral},
 a constructible $\Lambda$-complex $\widetilde{\mc{L}}$ on $X\times\widetilde{\mb{P}}'$,
 and a homomorphism $\alpha\colon(\sigma\times\mr{id}_{\widetilde{\mb{P}}'})^*\widetilde{\mc{L}}\rightarrow\widetilde{\mc{L}}$ such that the following hold:
 \begin{itemize}
  \item $\widetilde{\mc{L}}$ is an extension of $(\mc{L}_{\Lambda}(ft'))_U$ where $U:=\mb{A}^1_{t'}\times_{\mb{P}^1_{t'}}\widetilde{\mb{P}}'\subset\widetilde{\mb{P}}'$;

  \item $\alpha$ is an extension of the canonical isomorphism
	\begin{equation*}
	 (\sigma\times\mr{id})^*\bigl(\mc{L}_{\Lambda}(ft')\bigr)_U
	  \cong
	  \bigl(\sigma^*\mc{L}_{\Lambda}(ft')\bigr)_U
	  \cong
	  \bigl(\mc{L}_{\Lambda}(\sigma^*(f)t')\bigr)_U
	  \cong
	  \mc{L}_{\Lambda}(ft')_U
	\end{equation*}
	under the identification $\widetilde{\mc{L}}|_U\cong(\mc{L}_{\Lambda}(ft'))_U$;

  \item $\widetilde{\mc{L}}$ is $\widetilde{\mr{pr}}'$-universally locally acyclic,
	where $\widetilde{\mr{pr}}'\colon X\times_{\mb{P}^1_{t'}}\widetilde{\mb{P}}'\rightarrow\widetilde{\mb{P}}'$ is the projection.
 \end{itemize}
 We fix (geometric) points $\widetilde{0}'$ and $\widetilde{\infty}'$ of $\widetilde{\mb{P}}'$ over $0'$ and $\infty'$ of $\mb{P}^1_{t'}$,
 which exist since the morphism $\widetilde{\mb{P}}'\rightarrow\mb{P}^1_{t'}$ is taken to be surjective.
 Since $\mc{L}_{\Lambda}(ft')$ is $\mr{pr}'$-good because $\mb{P}^1_{t'}$ is $1$-dimensional and $\widetilde{\mc{L}}$ is $\widetilde{\mr{pr}}'$-locally acyclic,
 we have $\widetilde{\mc{L}}|_{X\times\widetilde{\infty}'}\cong\Psi_{\mr{pr}'}(\mc{L}_{\Lambda}(ft'))$.

 Under this setup, the relative trace $\tau_{\widetilde{\mb{P}}'}\bigl(\widetilde{\mc{L}},\alpha\bigr)$ is defined as an element of
 $\mr{H}^0(X^\sigma\times\widetilde{\mb{P}}',\mc{K}_{X^\sigma/\widetilde{\mb{P}}'})$.
 Since $\sigma\neq1$, we have $(X^{\sigma})_{\mr{red}}\cong\{x\}$ by assumption, and we have the following commutative diagram
 \begin{equation*}
  \xymatrix{
   \mr{H}^0(X^\sigma\times\widetilde{0}',\mc{K}_{X^\sigma/\widetilde{0}'})\ar[r]^-{\sim}&
   \mr{H}^0((X^\sigma)_{\mr{red}}\times\widetilde{0}',\mc{K}_{(X^\sigma)_{\mr{red}}/\widetilde{0}'})\ar@{-}[r]^-{\sim}&
   \mr{H}^0(\widetilde{0}',\Lambda)\ar@{-}[r]^-{\sim}&
   \Lambda\\
  \mr{H}^0(X^\sigma\times\widetilde{\mb{P}}',\mc{K}_{X^\sigma_{\widetilde{\mb{P}}'}/\widetilde{\mb{P}}'})
   \ar[r]^-{\sim}\ar[d]_{\mr{res}_{\widetilde{\infty}'}}\ar[u]^{\mr{res}_{\widetilde{0}'}}&
   \mr{H}^0((X^\sigma)_{\mr{red}}\times\widetilde{\mb{P}}',\mc{K}_{(X^\sigma_{\widetilde{\mb{P}}'})_{\mr{red}}/\widetilde{\mb{P}}'})\ar@{-}[r]^-{\sim}\ar[d]\ar[u]&
   \mr{H}^0(\widetilde{\mb{P}}',\Lambda)\ar@{-}[r]^-{\sim}_-{\heartsuit}\ar[d]\ar[u]&
   \Lambda\ar@{=}[d]\ar@{=}[u]\\
   \mr{H}^0(X^\sigma\times\widetilde{\infty}',\mc{K}_{X^\sigma/\widetilde{\infty}'})\ar[r]^-{\sim}&
   \mr{H}^0((X^\sigma)_{\mr{red}}\times\widetilde{\infty}',\mc{K}_{(X^\sigma)_{\mr{red}}/\widetilde{\infty}'})\ar@{-}[r]^-{\sim}&
   \mr{H}^0(\widetilde{\infty}',\Lambda)\ar@{-}[r]^-{\sim}&
   \Lambda,
   }
 \end{equation*}
 where $\mr{res}_{\star}$ is the restriction homomorphism.
 Note that the homomorphism marked with $\heartsuit$ is an isomorphism since we took $\widetilde{\mb{P}}'$ to be connected.
 We identify $\Lambda$-modules appearing in this diagram by these isomorphisms.
 Then we have an equality
 $\mr{res}_{\widetilde{0}'}\tau_{\widetilde{\mb{P}}'}\bigl(\widetilde{\mc{L}},\alpha\bigr)=
 \mr{res}_{\widetilde{\infty}'}\tau_{\widetilde{\mb{P}}'}\bigl(\widetilde{\mc{L}},\alpha\bigr)$
 in $\Lambda$.
 It remains to show the following 2 claims:
 \begin{enumerate}
  \item\label{thm-1}
       $\mr{res}_{\widetilde{0}'}\tau_{\widetilde{\mb{P}}'}\bigl(\widetilde{\mc{L}},\alpha\bigr)=-a_G(\sigma)$ in $\Lambda$;

  \item\label{thm-2}
       $\mr{res}_{\widetilde{\infty}'}\tau_{\widetilde{\mb{P}}'}\bigl(\widetilde{\mc{L}},\alpha\bigr)=\mr{Tr}\bigl(\sigma;\Psi_{\mr{pr}'}(\mc{L}_{\Lambda}(ft'))_{(x,\infty')}\bigr)$
       in $\Lambda$.
 \end{enumerate}
 Let us show \ref{thm-1}.
 Invoking Theorem \ref{bc}, we have
 \begin{equation*}
  \mr{res}_{\widetilde{0}'}\tau_{\widetilde{\mb{P}}'}\bigl(\widetilde{\mc{L}},\alpha\bigr)=
   \tau_{\widetilde{0}'}\bigl(\widetilde{\mc{L}}|_{X\times\widetilde{0}'},\alpha|_{X\times\widetilde{0}'}\bigr)=
   \tau_{\widetilde{0}'}(\Lambda,\mr{can}),
 \end{equation*}
 where $\mr{can}$ is the canonical isomorphism $\sigma^*\Lambda\cong\Lambda$.
 In view of \cite[Rem 2.15]{LZ} the categorical trace map $\tau_{\widetilde{0}'}$ coincides with that in [SGA 5, Exp.\ III, \S4]
 (or more precisely $\left<-,\mr{id}\right>$).
 By [SGA 5, Exp.\ III, Ex.\ 4.3], this coincides with the intersection product $(\Delta_X\cdot\Gamma_\sigma)_{X\times X}$,
 which is nothing but $\mr{length}(\mc{O}_{X^{\sigma}})=-a_G(\sigma)$ by \cite[Prop 7.1 (b)]{Fu}.
 Let us show \ref{thm-2}. Similarly, we have
 \begin{equation*}
  \mr{res}_{\widetilde{\infty}'}\tau_{\widetilde{\mb{P}}'}\bigl(\widetilde{\mc{L}},\alpha\bigr)=
   \tau_{\widetilde{\infty}'}\bigl(\widetilde{\mc{L}}|_{X\times\widetilde{\infty}'},\alpha|_{X\times\widetilde{\infty}'}\bigr)=
   \tau_{\widetilde{\infty}'}\bigl(\Psi_{\mr{pr}'}(\mc{L}_{\Lambda}(ft')),\sigma'\bigr),
 \end{equation*}
 where $\sigma'$ is the composition of canonical isomorphisms
 $\sigma^*\Psi_{\mr{pr}'}(\mc{L}_{\Lambda}(ft'))\cong\Psi_{\mr{pr}'}(\sigma^*\mc{L}_{\Lambda}(ft'))\cong\Psi_{\mr{pr}'}(\mc{L}_{\Lambda}(ft'))$.
 By Proposition \ref{complem}.\ref{complem-1} below (which is a reproduction of \cite[6.5.1]{A} and circular reasoning do not occur),
 $\Psi_{\mr{pr}'}(\mc{L}_{\Lambda}(ft'))$ is supported on $(x,\infty')$.
 Let $i\colon x\hookrightarrow X$ be the closed immersion, and put $P:=\Psi_{\mr{pr}'}(\mc{L}_{\Lambda}(ft'))_{(x,\infty')}$.
 Then we have $\Psi_{\mr{pr}'}(\mc{L}_{\Lambda}(ft'))\cong i_*P$.
 By [SGA 5, Exp.\ III, Thm 4.4], we have
 \begin{equation*}
  \tau_{\widetilde{\infty}'}\bigl(\Psi_{\mr{pr}'}(\mc{L}_{\Lambda}(ft')),\sigma'\bigr)
   =
   \tau_{\widetilde{\infty}'}(i_*P,i_*\sigma'_x)
   =
   \iota\tau_{\widetilde{\infty}'}(P,\sigma'_x),
 \end{equation*}
 where $\iota\colon\mr{H}^0(x,\mc{K}_{x/\widetilde{\infty}'})\xrightarrow{\sim}\mr{H}^0(X^{\sigma},\mc{K}_{X^{\sigma}/\widetilde{\infty}'})$.
 The categorical trace $\tau$ for $x$ is nothing but the ordinary trace for $\Lambda$-modules, and the claim follows.
\end{proof}

\section{$\mb{Q}_\ell$-rationality of Artin representation}
\label{sec2}
Arguing similarly to the proof of \cite[5.3]{KSS} with $\reptd(\mr{R}\phi_x\Lambda)$ in \cite[Conj 5.1]{KSS} replaced by $\Psi_{\mr{pr}'}(\mc{L}(ft'))_{(x,\infty')}$,
Conjecture \ref{conj-1} immediately follows thanks to Theorem \ref{thm}.
However, the best we can show with this argument is that the Artin representation is $\mb{Q}_\ell(\zeta_p)$-rational.
In order to achieve the $\mb{Q}_\ell$-rationality, namely to prove Conjecture \ref{conj-2},
we wish to compare $\reptd(\mr{R}\phi_x\Lambda)$ and $\Psi_{\mr{pr}'}(\mc{L}(ft'))_{(x,\infty')}$ as virtual $G$-representations,
which can be viewed as refinements of some results of \cite[\S6]{A}.

\subsection{}
For a triangulated category $\mc{T}$, we denote by $\mr{K}\mc{T}$ the Grothendieck group of $\mc{T}$.
For a ring $A$, let $D_{\mr{perf}}(A)$ the derived category of perfect $A$-complexes.
We denote by $\mr{K}^{\cdot}(A)$ the Grothendieck group $\mr{K}D_{\mr{perf}}(A)$, borrowing the notation from SGA 6.
For a finite group $G$ and a field $L$ over $\mb{Q}_\ell$, we denote $\mr{K}^{\cdot}(L[G])$ also by $\mr{R}_L(G)$ following \cite{Slin}.

\subsection{}
Let $k$ be an algebraically closed field of characteristic $p>0$, and let $K:=\mr{Frac}(k\{t\})$,
the function field of the (strict) henselization of $\mr{Spec}(k[t])$ at $t=0$.
Put $\Gamma:=\mr{Gal}(K^{\mr{sep}}/K)$, where $K^{\mr{sep}}$ is a separable closure of $K$.
For a finite dimensional $\mb{Q}_\ell$-vector space $M$ equipped with a continuous $\Gamma$-action,
we wish to construct the ``total dimension representation'' $\reptd(M)$.
The construction in \cite[(1.2.1)]{KSS} is slightly ambiguous\footnote{
The ``functor'' $\reptd$ is defined in \cite[\S1]{KSS}.
Let $L$ be a finite Galois extension of $K$, and let $M$ be a $(\mb{Z}/\ell^n)[\mr{Gal}(L/K)]$-module for some $n$.
In \cite{KSS}, they define $\repsw^{(L)}(M)$ to be $\mr{Sw}_{\mr{Gal}(L/K)}\otimes_{\mb{Z}_{\ell}[\mr{Gal}(L/K)]}M$ as we do here.
Then they claim that this does not depend on the choice of $L$.
It is true that $\repsw^{(L)}(M)$ is independent on the choice of $L$ in the sense that if we take another Galois extension $L'\supset L$ of $K$,
there is {\em some} isomorphism $\repsw^{(L')}(M)\cong\repsw^{(L)}(M)$ compatible with the map $\mr{Gal}(L'/K)\rightarrow\mr{Gal}(L/K)$.
However, this isomorphism is {\em not} canonical since $\mr{Sw}_{\mr{Gal}(L/K)}$ is merely a representation constructed from the Swan character.
Thus, for a finite $\mb{Z}_{\ell}$-module $M$ equipped with continuous $\Gamma$-action, it is not clear if $\{\repsw(M/\ell^n)\}$ forms an inverse system.
},
so let us make it more precise.
However, contrary to \cite{KSS}, we do not construct $\reptd$ as a functor, but we content ourselves
just by constructing $\reptd(M)$ when we are given $M$.

Let $L$ be a finite Galois extension of $K$.
Recall that the Swan representation $\mr{Sw}_{\mr{Gal}(L/K)}$ is defined.
This is a projective $\mb{Z}_\ell[\mr{Gal}(L/K)]$-module by \cite[19.2, Thm 44]{Slin}.
Now, let $K=K_0\subset K_1\subset K_2\subset\dots$ be a tower of finite extensions of $K$ such that the extension $K_n/K$ is Galois.
We put $\Delta_i:=\mr{Gal}(K_i/K)$.
For each $i\in\mb{N}$, we may {\em choose} an isomorphism
\begin{equation*}
 \varphi_i\colon
 \mr{Sw}_{\Delta_{i+1}}\otimes_{\mb{Z}_\ell[\Delta_{i+1}]}\mb{Z}_\ell[\Delta_i]
  \xrightarrow{\sim}
  \mr{Sw}_{\Delta_i}.
\end{equation*}

Let $\Lambda$ be a finite $\mb{Z}_\ell$-algebra.
We say that a finitely generated $\Lambda$-module $M$ equipped with continuous $\Gamma$-action is {\em trivialized by the tower $\{K_i\}$} if,
for any $n$, the $\Gamma$-action on $M\otimes^{\mb{L}}_{\mb{Z}_\ell}\mb{Z}/\ell^n$ factors through $\Gamma\twoheadrightarrow\Delta_i$ for some $i$
(possibly depending on $n$).
If we are given a finitely generated $\Lambda$-module $M$ equipped with continuous $\Gamma$-action, we may choose a tower $\{K_i\}$ by which $M$ is trivialized.

Let $M$ be a continuous $\Lambda[\Gamma]$-module trivialized by $\{K_i\}$.
If $\Lambda$ is a {\em torsion ring}, we put
\begin{equation*}
 \repsw_{i}(M):=\mr{Sw}_{\Delta_i}\otimes_{\mb{Z}_\ell[\Delta_i]}M
  \cong
  \mr{Sw}_{\Delta_i}\otimes^{\mb{L}}_{\mb{Z}_\ell[\Delta_i]}M
  \,
  \bigl(\cong\mr{Hom}_{\mb{Z}_\ell[\Delta_i]}(\mr{Sw}_{\Delta_i},M)\bigr)
\end{equation*}
for any $i>N$ for sufficiently large $N$.
The isomorphism $\varphi_i$ that we fixed induces an isomorphism $\repsw_{i+1}(M)\xrightarrow{\sim}\repsw_i(M)$.
With this isomorphism, we identify $\repsw_{i}(M)$ for any $i>N$, and this $\Lambda[\Gamma]$-module is denoted by $\repsw_{\Lambda}(M)$.
We also put $\reptd_{\Lambda}(M):=\repsw_{\Lambda}(M)\oplus M$.
Note that the construction of $\repsw_{\Lambda}(M)$ is functorial, and furthermore exact, with respect to $\Lambda[\Gamma]$-module $M$ trivialized by $\{K_i\}$.
In particular, if we are given a finite group $G$ and a $\Lambda[\Gamma\times G]$-module $M$ such that
the underlying $\Lambda[\Gamma]$-module is trivialized by $\{K_i\}$, then $\repsw_{\Lambda}(M)$ and $\reptd_{\Lambda}(M)$ are $\Lambda[\Gamma\times G]$-modules.

Now, let $\Lambda$ be a finite {\em projective} $\mb{Z}_\ell$-algebra, and $G$ be a finite group.
Let $M$ be a finitely generated $\Lambda$-module equipped with continuous $\Gamma\times G$-action whose underlying $\Gamma$-action is trivialized by $\{K_i\}$.
Then we define
\begin{equation*}
 \repsw_{\Lambda}(M):=\invlim_n\repsw_{\Lambda/\ell^n}\bigl(M\otimes\mb{Z}/\ell^n\bigr).
\end{equation*}
If we wish to be more precise, we also denote it by $\repsw_{\Lambda}^{\{K_i\}}(M)$.
Similarly, $\repsw$ is a functor from the category of finitely generated $\Lambda$-modules equipped with continuous $\Gamma\times G$-action which
is trivialized by $\{K_i\}$ to the category of continuous $\Lambda[\Gamma\times G]$-modules.
One can also check easily that this functor is exact.
We also put $\reptd_{\Lambda}(M):=\repsw_{\Lambda}(M)\oplus M$.

\subsection{}
\label{tddfn}
Let $L$ be a field of finite extension of $\mb{Q}_\ell$.
Let $\mc{O}_L$ the ring of integers of $L$.
Finally, let $M$ be a finitely generated $L$-vector space equipped with continuous $\Gamma\times G$-action.
In this situation, let us define $\repsw_L(M)$ in $\mr{R}_L(G)$.
By continuity of the action, we may take an $\mc{O}_L$-lattice $N\subset M$ which is stable under $\Gamma\times G$-action.
We take a tower $\{K_i\}$ by which trivializes $N$.
Then we have:

\begin{lem*}
 The $G$-representation $\repsw^{\{K_i\}}_{\mc{O}_L}(N)\otimes\mb{Q}$ considered as an element in $\mr{R}_L(G)$
 does not depend on the choice of $N$ and $\{K_i\}$.
\end{lem*}
\begin{proof}
 Assume we are given another tower $\{K'_i\}$ which trivializes $N$.
 We wish to show that $\repsw^{\{K_i\}}_{\mc{O}_L}(N)\otimes\mb{Q}=\repsw^{\{K'_i\}}_{\mc{O}_L}(N)\otimes\mb{Q}$ in $\mr{R}_L(G)$.
 To show this, it suffices to show that $\mr{Tr}(\sigma;-)$ are the same for any $\sigma\in G$.
 Since $\mr{Tr}\bigl(\sigma;\repsw_{\mc{O}_L}(N)\bigr)=\bigl\{\mr{Tr}(\sigma;\repsw_{\mc{O}_L/\ell^n}(N\otimes\mb{Z}/\ell^n))\bigr\}_n$,
 it suffices to show that
 $\mr{Tr}(\sigma;\repsw^{\{K_i\}}_{\mc{O}_L/\ell^n}(N\otimes\mb{Z}/\ell^n))=\mr{Tr}(\sigma;\repsw^{\{K'_i\}}_{\mc{O}_L/\ell^n}(N\otimes\mb{Z}/\ell^n))$.
 This is true since $\mr{Sw}_{\Delta}\otimes_{\mb{Z}_\ell[\Delta]}N/\ell^n$ is (non-canonically!) isomorphic for any large enough $\Delta$.
 Let us show that it does not depend on the choice of $N$.
 If we are given another lattice $N'$, $N+N'\subset M$ is another lattice as well.
 By replacing $N'$ by $N+N'$, we may assume that $N'\supset N$.
 We fix a tower $\{K_i\}$ which trivializes $N'$.
 Then this also trivializes $N$.
 Consider the exact sequence $0\rightarrow N\rightarrow N'\rightarrow N'/N\rightarrow 0$.
 Since the functor $\repsw^{\{K_i\}}_{\mc{O}_L}$ is exact, it suffices to show that $\repsw^{\{K_i\}}_{\mc{O}_L}(N'/N)\otimes\mb{Q}=0$ in $\mr{R}_L(G)$.
 This holds since $N'/N$ is a torsion module.
\end{proof}

\begin{dfn*}
 Let $M$ be a finitely generated $L$-vector space equipped with continuous $\Gamma\times G$-action.
 \begin{enumerate}
  \item The representation $\repsw^{\{K_i\}}_{\mc{O}_L}(N)\otimes\mb{Q}$ considered as an element of $\mr{R}_L(G)$ is denoted by $\repsw(M)$.

  \item We put $\reptd(M):=\repsw(M)+M$ in $\mr{R}_L(G)$.
 \end{enumerate}
\end{dfn*}

\subsection{}
\label{trdfn}
Let $G$ be a finite group and $\Lambda$ be a commutative noetherian ring.
We define $D_{\Lambda\mr{-perf}}(\Lambda[G])$ to be the full subcategory of $D(\Lambda[G])$ spanned by complexes $\mc{F}$ which are $\Lambda$-perfect.
For $\sigma\in G$, the multiplication by $\sigma$ from the left induces the homomorphism
$\phi_\sigma\colon\Lambda[G]\rightarrow\Lambda[G]$ of {\em right} $\Lambda[G]$-modules.
For $\mc{F}\in D_{\Lambda\mr{-perf}}(\Lambda[G])$ and $\sigma\in G$,
we have the morphism
\begin{equation*}
 \varphi_\sigma\colon
  \mc{F}\cong\Lambda[G]\otimes^{\mb{L}}_{\Lambda[G]}\mc{F}
  \xrightarrow{\phi_\sigma\otimes\mr{id}}
  \Lambda[G]\otimes^{\mb{L}}_{\Lambda[G]}\mc{F}
  \cong
  \mc{F}
\end{equation*}
in $D_{\mr{perf}}(\Lambda)$.
We put $\mr{Tr}(\sigma;\mc{F}):=\mr{Tr}(\varphi_\sigma;\mc{F})$ (cf.\ [SGA 6, Exp.\ I, 8.1]).
Fixing $\sigma\in G$, the map $\mr{Tr}(\sigma;-)$ is additive\footnote{
It is well-known to the experts that the trace for the derived category of perfect complexes is non-additive in general.
This was first observed by D. Ferrand, and the interesting history of the discovery of this defect can be found in \cite{F}.
}.
The verification is elementary as follows:
Assume we are given an exact triangle $M'\xrightarrow{\alpha}M\rightarrow M''\rightarrow$ in $D_{\Lambda\mr{-perf}}(\Lambda[G])$.
We may find the following commutative diagram
\begin{equation*}
 \xymatrix{
  M'\ar[r]^-{\alpha}\ar[d]_{\mr{qis}}&M\ar[d]^{\mr{qis}}\\
 N'^{\bullet}\ar[r]^-{\beta}&N^{\bullet}}
\end{equation*}
such that $N^{(\prime)i}$ is finitely generated (left) $\Lambda[G]$-module which is projective over $\Lambda$ for any $i$,
$N^{(\prime)i}=0$ for all but finitely many $i$, and $\beta$ is an actual morphism of complexes.
Since the trace is additive for the {\em category of complexes} $C^{\mr{b}}(\Lambda[G])$, we have
\begin{equation*}
 \mr{Tr}\bigl(\sigma;N'^{\bullet}\bigr)+\mr{Tr}\bigl(\sigma;\mr{cone}(\beta)\bigr)=\mr{Tr}\bigl(\sigma;N^{\bullet}\bigr),
\end{equation*}
and the additivity follows.
Thus, we have the homomorphism $\mr{Tr}\colon\mr{K}(D_{\Lambda\mr{-perf}}(\Lambda[G]))\rightarrow\Lambda$.

\begin{prop}
 \label{complem}
 Consider the situation of Theorem {\normalfont\ref{mainthmcons}}.
 \begin{enumerate}
  \item\label{complem-1}
       The complex $\Psi_{\mr{pr}'}(\mc{L}(ft'))$ on $X\times\infty'$ is supported on $(x,\infty')$.

  \item\label{complem-2}
       We have $\reptd((\mr{R}\phi_f\Lambda)_x)=-\Psi_{\mr{pr}'}(\mc{L}(ft'))_{(x,\infty')}$ in $\mr{R}_{\mb{Q}_\ell(\zeta_p)}(G)$.
 \end{enumerate}
\end{prop}
\begin{proof}
 We have already proved \ref{complem-1} in \cite[6.5.1]{A}, but we include the proof for the sake of completeness.
 Put $S:=\mr{Spec}(k\{t\})$, $S':=\mr{Spec}(k\{1/t'\})$, and let $q'\colon S\times S'\rightarrow S'$ be the projection.
 Moreover, let $\eta$ be the generic point of $S$ and $0$ (resp.\ $\infty'$) be the closed point of $S$ (resp.\ $S'$).
 Let $M$ be a finite dimensional $\Lambda:=\mb{Q}_\ell(\zeta_p)$-vector space on which $\pi_1(\eta)\times G$ acts continuously.
 Then we claim that
 \begin{equation}
  \label{repeq}\tag{$\star$}
  \reptd_{\Lambda}(M)=
  -\left[\Psi_{q'}\bigl(\mc{L}_{\Lambda}(tt')\otimes_{\Lambda}\mc{M}_!\bigr)_{(0,\infty')}\right]
 \end{equation}
 in $\mr{R}_{\Lambda}(G)\cong\mr{K}^{\cdot}(\Lambda[G])$,
 where $\mc{M}_!$ is the zero extension of $\mc{M}$, the sheaf of (left) $\Lambda[G]$-modules on $\eta$ associated to $M$,
 by the open immersion $\eta\hookrightarrow S$.
 It suffices to show that for any $\sigma\in G$, we have
 \begin{equation}
  \label{treq}\tag{$\star\star$}
  \mr{Tr}(\sigma;\reptd_{\Lambda}(M))=
  -\mr{Tr}\bigl(\sigma;\Psi_{q'}(\mc{L}_\Lambda(tt')\otimes\mc{M}_!)_{(0,\infty')}\bigr)
 \end{equation}
 in $\Lambda$.
 For this, it suffices to show (\ref{treq}) when $\Lambda$ is a local regular finite $\mb{Z}_\ell$-algebra
 and $M$ is a finitely generated free $\Lambda$-module over which $\pi_1(\eta)\times G$ acts continuously.
 For $N\in D_{\Lambda\mr{-perf}}(\Lambda[G])$, for which the notion of trace makes sense by \ref{trdfn},
 we have $\mr{Tr}(\sigma;N)=\{\mr{Tr}(\sigma;N\otimes^{\mb{L}}_\Lambda\Lambda/\ell^n\Lambda)\}_{n\geq0}$
 via the identification $\Lambda\cong\invlim\Lambda/\ell^n\Lambda$.
 Thus, it suffices to show (\ref{treq}) when $\Lambda$ is a finite local ring with residual character $\ell$, and for this $\Lambda$,
 it suffices to show the identity (\ref{repeq}) in $\mr{K}(D_{\Lambda\mr{-perf}}(\Lambda[G]))$.

 Before we proceed, we make an observation.
 Let $\Delta$ be a finite group, $N'\rightarrow N\rightarrow N''\rightarrow$ be an exact triangle in $D^-(\Lambda[\Delta]^{\mr{op}})$,
 and let $M\in D^-(\Lambda[\Delta]\otimes_{\Lambda}\Lambda[G])$.
 Then we have an exact triangle
 \begin{equation*}
  N'\otimes^{\mb{L}}_{\Lambda[\Delta]}M\rightarrow N\otimes^{\mb{L}}_{\Lambda[\Delta]}M\rightarrow N''\otimes^{\mb{L}}_{\Lambda[\Delta]}M
   \rightarrow
 \end{equation*}
 in $D(\Lambda[G])$.
 In particular, if we are given $N,N'\in D_{\mr{perf}}(\Lambda[\Delta]^{\mr{op}})$ and $M\in D_{\Lambda\mr{-perf}}(\Lambda[\Delta\times G])$
 such that $[N]=[N']$ in $\mr{K}^{\cdot}(\Lambda[\Delta]^{\mr{op}})$,
 we have $[N\otimes^{\mb{L}}_{\Lambda[\Delta]}M]=[N'\otimes^{\mb{L}}_{\Lambda[\Delta]}M]$ in $\mr{K}(D_{\Lambda\mr{-perf}}(\Lambda[G]))$.

 Let $f\colon \xi\rightarrow\eta$ be a finite Galois extension with Galois group $\Delta$.
 We must show (\ref{repeq}) for a finite (commutative) ring $\Lambda$ in which $p$ is invertible
 and $M$ a finitely generated $\Lambda[\Delta\times G]$-module which is free as the underlying $\Lambda$-module.
 We view $\mc{D}:=f_*\Lambda_{\xi}$ as a sheaf of {\em right} $\Lambda[\Delta]$-modules on $\eta$.
 Then, the sheaf $\mc{L}_{\Lambda}(tt')\otimes^{\mb{L}}_{\Lambda}\mc{D}_!$ belongs to $D_{\mr{ctf}}(S\times S',\Lambda[\Delta]^{\mr{op}})$.
 Using [SGA 4, Exp.\ XVII, 5.2.11], $\Psi_{q'}\bigl(\mc{L}_{\Lambda}(tt')\otimes^{\mb{L}}_{\Lambda}\mc{D}_!\bigr)$
 belongs to $D_{\mr{ctf}}(S_{\infty'},\Lambda[\Delta]^{\mr{op}})$,
 which implies that $\Psi_{q'}\bigl(\mc{L}(tt')\otimes^{\mb{L}}_{\Lambda}\mc{D}_!\bigr)_{(0,\infty')}$ is a perfect $\Lambda[\Delta]^{\mr{op}}$-complex.
 Invoking \cite[2.6.1]{L} and \cite[2.4.2.2]{L}, we have an equality
 $\mr{Dt}_{\Delta}\otimes_{\mb{Z}_\ell}\Lambda=\Psi_{q'}\bigl(\mc{L}_{\Lambda}(tt')\otimes^{\mb{L}}_{\Lambda}\mc{D}_!\bigr)_{(0,\infty')}[-1]$
 in $\mr{K}^{\cdot}(\Lambda[\Delta]^{\mr{op}})$, where $\mr{Dt}_{\Delta}:=\mr{Sw}_\Delta\oplus\mb{Z}_\ell[\Delta]$.
 Now, for $M\in D_{\Lambda\mr{-perf}}(\Lambda[\Delta\times G])$, let $\mc{M}$ be the associated complex in $D^{\mr{b}}(\eta,\Lambda[G])$,
 and denote by $\mc{M}_!$ the zero extension to $S$.
 For a scheme $X$, finite groups $A$, $B$, and complexes $\mc{F}\in D^-(X,\Lambda[A]^{\mr{op}})$ and $M\in D^-(\Lambda[A\times B])$,
 we simply denote by $\mc{F}\otimes^{\mb{L}}_{\Lambda[A]}M$ in $D^-(X,\Lambda[B])$ the tensor product $\mc{F}\otimes^{\mb{L}}_{\Lambda[A]}\pi^*M$
 where $\pi\colon X_{\mr{\acute{e}t}}\rightarrow\Set$ is the unique morphism of topoi.
 Using this notation, we have $\mc{D}_!\otimes^{\mb{L}}_{\Lambda[\Delta]}M\cong\mc{M}_!$.
 We have the isomorphisms in $D_{\Lambda\mr{-perf}}(\Lambda[G])$:
 \begin{align*}
  \reptd(M)
   &\cong
  \mr{Dt}_{\Delta}\otimes^{\mb{L}}_{\mb{Z}_\ell[\Delta]}M,\\
  \Psi_{q'}\bigl(\mc{L}_{\Lambda}(tt')\otimes^{\mb{L}}_{\Lambda}\mc{M}_!\bigr)
  &\cong
  \Psi_{q'}\bigl(\mc{L}_{\Lambda}(tt')\otimes^{\mb{L}}_{\Lambda}\mc{D}_!\bigr)
  \otimes^{\mb{L}}_{\Lambda[\Delta]}M.
 \end{align*}
 Indeed, the first isomorphism follows by definition.
 Let us check the second one.
 Put $\mc{H}^0\Psi_{q'}=:\Psi_{q'}^0$.
 Note that the cohomological dimension of $\Psi_{q'}$ is bounded by \cite[Prop 3.1]{O}
 and $\Psi_{q'}\cong\mr{R}\Psi_{q'}^0$ by definition.
 This implies that we may take a quasi-isomorphism $\mc{L}_{\Lambda}(tt')\otimes_{\Lambda}\mc{D}_!\simeq\mc{I}^{\bullet}$
 of right $\Lambda[\Delta]$-complexes on $\eta$ such that
 $\mc{I}^n=0$ for almost all $n\in\mb{Z}$ and $\mc{I}^n$ is $\Psi^0_{q'}$-acyclic for any $n$.
 Now, let $P_{\bullet}\rightarrow M$ be a resolution (not necessarily bounded below) by finite free $\Lambda[\Delta\times G]$-modules.
 We have quasi-isomorphisms
 \begin{align*}
  \Psi_{q'}\bigl(\mc{L}_{\Lambda}(tt')\otimes^{\mb{L}}_{\Lambda}\mc{M}_!\bigr)
  &\simeq
  \Psi_{q'}\bigl(\mc{L}_{\Lambda}(tt')\otimes^{\mb{L}}_{\Lambda}\mc{D}_!\otimes^{\mb{L}}_{\Lambda[\Delta]}M\bigr)
  \simeq
  \Psi^0_{q'}(\mc{I}^{\bullet}\otimes_{\Lambda[\Delta]}P_{\bullet})\\
  &\simeq
  \Psi^0_{q'}(\mc{I}^{\bullet})\otimes_{\Lambda[\Delta]}P_{\bullet}
  \simeq
  \Psi_{q'}\bigl(\mc{L}_{\Lambda}(tt')\otimes^{\mb{L}}_{\Lambda}\mc{D}_!\bigr)
  \otimes^{\mb{L}}_{\Lambda[\Delta]}M,
 \end{align*}
 where the 2nd quasi-isomorphism follows by \cite[14.3.4]{KSc}.
 Thus, the claim follows.

 By the observation above, we have the equality $\reptd(M)=-\left[\Psi_{q'}\bigl(\mc{L}_{\Lambda}(tt')_!\otimes^{\mb{L}}_{\Lambda}M\bigr)_{(0,\infty')}\right]$
 in $\mr{K}(D_{\Lambda\mr{-perf}}(\Lambda[G]))$.
 Thus, by \ref{trdfn}, (\ref{repeq}) holds when $\Lambda$ is a finite ring, and (\ref{repeq}) holds when $\Lambda=\mb{Q}_\ell(\zeta_p)$ as well.
 Now, let $\mc{F}$ be a constructible $\Lambda[G]$-module on $S$.
 By considering the localization sequence, (\ref{repeq}) implies that
 \begin{equation*}
  \tag{$\heartsuit$}
  \reptd(\mr{R}\phi_{\mr{id}}(\mc{F}))=
   -\left[\Psi_{q'}\bigl(\mc{L}(tt')\otimes_{\Lambda}q^*\mc{F}\bigr)_{(0,\infty')}\right],
 \end{equation*}
 where $q\colon S\times S'\rightarrow S$ is the projection.
 For the rest of the argument, we only need to copy the proof of \cite[6.5]{A}.
 However, we reproduce the proof for the convenience of the reader.

 Let $y$ be a geometric point of $X$ over $0\in S$, and consider the following commutative diagram of strictly henselian schemes:
 \begin{equation*}
  \xymatrix@C=70pt{
   (X\times S')_{(y,\infty')}\ar[r]^-{g:=(f\times\mr{id})_{(y,\infty')}}\ar[d]&
   (S\times S')_{(0,\infty')}\ar[r]^-{\pi}\ar[d]^{q_{(0,\infty')}}&
   S'\\
  X_{(y)}\ar[r]^-{f_{(y)}}&S.&
   }
 \end{equation*}
 We have
 \begin{align*}
  \tag{$\heartsuit$}
  \Psi_{\mr{pr}'}(\mc{L}(ft'))_{(y,\infty')}
   &\cong
   \mr{R}(\pi\circ g)_*(\mc{L}(ft'))_{\overline{\eta}'}
   \cong
   \mr{R}(\pi\circ g)_*g^*(\mc{L}(tt'))_{\overline{\eta}'}
   \cong
   \mr{R}\pi_*\bigl(\mc{L}(tt')\otimes\mr{R}g_*\Lambda\bigr)_{\overline{\eta}'}\\
  &\cong
  \mr{R}\pi_*\bigl(\mc{L}(tt')\otimes q_{(0,\infty')}^*\mr{R}f_{(y)*}\Lambda\bigr)_{\overline{\eta}'}
  \cong
  \Psi_{q'}\bigl(\mc{L}(tt')\otimes_{\Lambda}q^*\mr{R}f_{(y)*}\Lambda\bigr)_{(0,\infty')},
 \end{align*}
 where the 1st and the last isomorphism follows by definition,
 the 3rd by ``projection formula'' (cf.\ \cite[Lem 6.4]{A}) because $\mc{L}(tt')$ is smooth on the fiber of $\eta'$,
 the 4th by the compatibility of the nearby cycle and the base change (in other words the $f_{(y)}$-goodness of any $\Lambda$-module on $X_{(y)}$).
 If $y$ does not lie on $x$, then since $f$ is smooth outside of $x$, $f$ is locally acyclic at $y$.
 In particular, $\mr{R}f_{(y)*}\Lambda\cong\Lambda$.
 Thus, by the above isomorphism together with Katz-Laumon vanishing, we get \ref{complem-1}.
 Finally, we note that for any $\Lambda$-module $\mc{F}$ on $X$, we have
 \begin{equation*}
  \tag{$\heartsuit$}
  \mr{R}\phi_{\mr{id}}\bigl(\mr{R}f_{(y)*}\mc{F}\bigr)
   \cong
   \mr{R}\phi_f(\mc{F})_{y}.
 \end{equation*}
 We combine 3 displayed equalities $(\heartsuit)$ to conclude the proof of \ref{complem-2}.
\end{proof}

\subsection{Proof of Main theorem}\mbox{}\\
By \cite[Lem 5.3]{KSS}, it suffices to show \cite[Conj 5.1]{KSS} in the case where $S=k\{t\}$ with $k$ an algebraically closed field of characteristic $p$.
This follows by combining Proposition \ref{complem} and Theorem \ref{thm}.
\qed

\section{Concluding remarks}

\subsection{}
There are at least two geometric constructions of the classical Artin representation in the literature:
one by Laumon \cite[2.6.1]{L} that we have used, and the other by Katz \cite[1.6]{K}.
These two constructions are not totally independent, and the idea of the proof of Theorem \ref{mainthmcons} clarifies a relation.
Let $k$ be a separably closed field and put $K:=\mr{Frac}(k\{t\})$ as before.
Let $f\colon\mr{Spec}(L)\rightarrow\mr{Spec}(K)$ be a finite Galois extension with Galois group $G$.
We put $\mr{Reg}_G:=f_*\mb{Z}_\ell$, and take a smooth sheaf $(\mr{Reg}_G)^{\sim}$ on $\mb{G}_{\mr{m},t}$ ({\it e.g.}, canonical extension)
such that it coincides with $\mr{Reg}_G$ around $0$ and tame at $\infty$.
Let $\mb{G}_{\mr{m}}\xrightarrow{j}\mb{A}^1\xrightarrow{j'}\mb{P}^1$ be the open immersions.
Katz showed that $\mr{H}^1\bigl(\mb{P}^1,j'_!\mr{R}j_*(\mr{Reg}_G)^{\sim}\bigr)$ viewed as a $\mb{Z}_\ell[G]$-module realizes the Swan representation.
Consider $\mc{M}:=j'_!\mr{R}j_*(\mr{Reg}_G)^{\sim}\otimes\mc{L}(tt')$ on $\mb{P}^1_t\times\mb{P}^1_{t'}$ using the notations of Theorem \ref{mainthmcons}.
We consider $\mc{C}:=\mr{R}\mr{pr}'_*\mc{M}$ on $\mb{P}^1_{t'}$, where $\mr{pr}'\colon\mb{P}^1_t\times\mb{P}^1_{t'}\rightarrow\mb{P}^1_{t'}$ is the projection.
Invoking \cite[2.4.3 (ii)]{L}, the vanishing cycles of $\mr{R}\mr{pr}'_*\mc{M}$ around $0$ is equal to $\Lambda[G]$ as virtual representations of $G$.
For a geometric point $x$ of $\mb{P}^1_{t'}$ over a closed point,
let $\Psi_{\mr{id},x\leftarrow\overline{\eta}'}$ denote the nearby cycle functor of $\mr{id}$ around $x$, using the notation of \cite[\S4]{A}.
Since $\mr{H}^1(\mb{P}^1,\mc{M}_{0'})\cong\mc{C}_{0'}$, we have
\begin{equation*}
 \left[\Psi_{\mr{id},0'\leftarrow\overline{\eta}'}(\mc{C})\right]=\left[\mr{H}^1(\mb{P}^1,\mc{M}_{0'})\right]+[\Lambda[G]]
\end{equation*}
as virtual representations of $G$.
Note that $\left[\mr{H}^1(\mb{P}^1,\mc{M}_{0'})\right]$ is nothing but Katz's Swan representation.
On the other hand, in view of \cite[2.4.3 (iii)]{L}, we have
\begin{equation*}
 \Psi_{\mr{id},\infty'\leftarrow\overline{\eta}'}(\mc{C})
  \cong
  \Psi_{\mr{pr}',\infty'\leftarrow\overline{\eta}'}\bigl(j'_!\mr{R}j_*(\mr{Reg}_G)^{\sim}\otimes\mc{L}(tt')\bigr)_{(0,\infty')},
\end{equation*}
and the latter module is a variant of the one appeared in Laumon's construction.
Since $\Psi_{x\leftarrow\overline{\eta}'}(\mc{C})$ does not depend on a geometric point $x$ of $\mb{P}^1_{t'}$ as a virtual representation of $G$,
Katz's and Laumon's construction are related via $\mc{C}$.

\subsection{}
Serre's conjecture do not tell us about the rationality of the Artin representation over $\mb{Q}_p$.
If fact, even in the classical case, the Artin representation is not always $\mb{Q}_p$-rational as shown in \cite[\S5]{S}.
However, it should be reasonable to expect the following:

\begin{conj*}
 We consider the situation of Introduction, and assume that the residue field $k:=A/\mf{m}$ is perfect.
 Then the Artin representation is $W(k)\otimes\mb{Q}$-rational.
\end{conj*}

The case where $A$ is a discrete valuation ring, the conjecture is known due to J.-M.\ Fontaine.
We expect that the argument in this article can be adapted to the $p$-adic cohomology theory using the nearby cycle functor of \cite{Ane}.

\subsection{}
In Theorem \ref{mainthmcons}, the case $\sigma=1$ is excluded.
The right hand side of (\ref{nbcyag}) is equal to $-\bigl(\mr{d}f,[T^*_XX]\bigr)_{T^*X}$ by the Milnor formula (cf.\ \cite[Thm 5.9]{Schar}).
This is nothing but the (negative of) Milnor number $-\mu_f$ of $f$ by \cite[Ex 14.1.5 (d)]{Fu}.
Note that, {\it a priori}, $\mu_f$ depends on $f$, and we cannot expect to coincide with $-a_{X,G}(1)$ without some constraints on $f$.

In the case where $X$ is of dimension $1$, Takeshi Saito observed that the validity of (\ref{nbcyag}) is related to the conductor-discriminant formula.
In this case, $\mu_f=(\mr{d}f,T^*_XX)$ is nothing but the discriminant $\mr{disc}(f)$ of $f$ (at $f(x)$).
Since $X/G$ is normal, it is smooth.
The morphism $f$ factors as $X\xrightarrow{q} X/G\xrightarrow{g} S$, and we have $\mr{disc}(q)+\mr{disc}(g)=\mr{disc}(f)$.
On the other hand, the conductor-discriminant formula implies that $\mr{disc}(q)=a_{G,X}(1)$.
Thus, we have $\mu_f-a_{G,X}(1)=\mr{disc}(g)$, and in particular, the equality (\ref{nbcyag}) remains to hold for $\sigma=1$ if $g$ is \'{e}tale at $q(x)$.
It might be interesting to generalize this picture to higher dimensional situation.

\noindent
Tomoyuki Abe:\\
Kavli Institute for the Physics and Mathematics of the Universe (WPI)\\
The University of Tokyo\\
5-1-5 Kashiwanoha,  
Kashiwa, Chiba, 277-8583, Japan\\
e-mail: {\tt tomoyuki.abe@ipmu.jp}

\end{document}